\documentclass[12pt]{article}
\usepackage{a4wide}
\usepackage{amsmath,amsfonts,amssymb,amsthm,amscd}
 
\usepackage{MnSymbol} %for \cupdot
\usepackage{hyperref}
\usepackage{xspace}

\usepackage{graphicx}
\usepackage{enumitem}

\theoremstyle{plain}

\newtheorem{theorem}{Theorem} %[section]
\newtheorem{lemma}[theorem]{Lemma}

\newtheorem{corollary}[theorem]{Corollary}
\newtheorem{proposition}[theorem]{Proposition}
\newtheorem{observation}[theorem]{Observation}
\newtheorem{conjecture}[theorem]{Conjecture}

\def\eps{\epsilon}
\def\DT{\mbox{\ensuremath{\mathcal{DT}}}\xspace}
\mathchardef\mhyphen="2D 
\def\R{\mbox{\ensuremath{\mathbb R}}\xspace}
\def\F{\mbox{\ensuremath{\mathcal F}}\xspace}
\def\HH{\mbox{\ensuremath{\mathcal H}}\xspace} 

\def\E{\mbox{\ensuremath{\mathcal E}}\xspace}
\def\B{\mbox{\ensuremath{\mathcal B}}\xspace}
\def\P{\mbox{\ensuremath{\mathcal P}}\xspace}

\mathcode`l="8000
\begingroup
\makeatletter
\lccode`\~=`\l
\DeclareMathSymbol{\lsb@l}{\mathalpha}{letters}{`l}
\lowercase{\gdef~{\ifnum\the\mathgroup=\m@ne \ell \else \lsb@l \fi}}%
\endgroup

\begin{document}
\title{Proper Coloring of Geometric Hypergraphs\footnote{
		Research supported by the Lend\"ulet program of the Hungarian Academy of Sciences (MTA), under grant number LP2017-19/2017. First author was also supported by Hungarian National Science Fund (OTKA), under grant PD 108406 and by the J\'anos Bolyai Research Scholarship of the Hungarian Academy of Sciences. Second author was also supported by the Marie Sk\l odowska-Curie action of the EU, under grant IF 660400.}}

\author{Bal\'azs Keszegh\thanks{Alfr\'ed R\'enyi Institute of Mathematics, Hungarian Academy of Sciences,	Budapest, Hungary and MTA-ELTE Lend\"ulet Combinatorial Geometry Research Group, Institute of Mathematics, E\"otv\"os Lor\'and University (ELTE), Budapest, Hungary.
	\texttt{keszegh@renyi.hu}}
\and D\"om\"ot\"or P\'alv\"olgyi\thanks{MTA-ELTE Lend\"ulet Combinatorial Geometry Research Group, Institute of Mathematics, E\"otv\"os Lor\'and University (ELTE), Budapest, Hungary
	%Centre for Mathematical Sciences, Wilberforce Road, Cambridge CB3 0WB
	\texttt{dom@cs.elte.hu}}}

\date{}

\maketitle

\begin{abstract}
We study whether for a given planar family \F there is an $m$ such that any finite set of points can be $3$-colored such that any member of \F that contains at least $m$ points contains two points with different colors.
We conjecture that if \F is a family of pseudo-disks, then such an $m$ exists.
We prove this in the special case when \F is the family of all homothetic copies of a given convex polygon.
We also study the problem in higher dimensions.
\end{abstract}

%%%%%%%%%%%%%%%%%%%%%%%%%%%%%%%%%%%%%%%%%%%%%%%%%%%%%%%%%%%%%%%%%%%%
\section{Introduction}\label{sec:Intro}
%%%%%%%%%%%%%%%%%%%%%%%%%%%%%%%%%%%%%%%%%%%%%%%%%%%%%%%%%%%%%%%%%%%%

In the present paper, we primarily focus on the following proper coloring problem.
Given a finite set of points in the plane, $S$, we want to color the points of $S$ with a small number of colors such that every member of some given geometric family \F that intersects $S$ in many points will contain at least two different colors.

Pach conjectured in 1980 \cite{P80} that for every convex set $D$ there is an $m$ such that any finite set of points admits a $2$-coloring such that any translate of $D$ that contains at least $m$ points contains both colors.
This conjecture inspired a series of papers studying the problem and its variants - for a recent survey, see \cite{PPT}.
Eventually, the conjecture was shown to hold in the case when $D$ is a convex polygon in a series of papers \cite{P86,TT07,PTconvex}, but disproved in general \cite{PP}.
In fact, the conjecture fails for any $D$ with a smooth boundary, e.g., for a disk.

It follows from basic properties of {generalized Delaunay triangulations} (to be defined later) and the Four Color Theorem that for any convex $D$ it is possible to $4$-color any finite set of points such that any {\em homothetic copy}\footnote{A homothetic copy or homothet of a set is a scaled and translated copy of it (rotations are {\em not} allowed).} of $D$ that contains at least two points will contain at least two colors.
Therefore, the only case left open is when we have $3$ colors.
We conjecture that for $3$ colors the following holds.

\begin{conjecture}\label{conj:homothet} For every plane convex set $D$ there is an $m$ such that any finite set of points admits a $3$-coloring such that any homothetic copy of $D$ that contains at least $m$ points contains two points with different colors.
\end{conjecture}

The special case of Conjecture \ref{conj:homothet} when $D$ is a disk has been posed earlier by the first author \cite{Khalf}, and is also still open.
Our main result is the proof of Conjecture \ref{conj:homothet} for convex polygons.

\begin{theorem}\label{thm:main3} For every convex $n$-gon $D$ there is an $m$ such that any finite set of points admits a $3$-coloring such that any homothetic copy of $D$ that contains at least $m$ points contains two points with different colors.
\end{theorem}

We would like to remark that the constructions from \cite{PP} do not exclude the possibility that for convex polygons the strengthening of Theorem \ref{thm:main3} using only $2$ colors instead of $3$ might also hold; this statement is known to hold for triangles \cite{KP11} and squares\footnote{And since affine transformation have no effect on the question, also for parallelograms.} \cite{AKV}.

The constant $m$ which we get from our proof depends not only on the number of sides, but also on the shape of the polygon.
However, we conjecture that this dependence can be removed, and in fact the following stronger conjecture holds for any {\em pseudo-disk arrangement}.
We define a pseudo-disk arrangement as a family of planar bodies whose boundaries are Jordan curves such that any member of the family intersects the boundary of any other member in a connected curve.\footnote{This is slightly non-standard, as usually it is assumed that any two boundaries intersect at most twice (the structure of the boundary curves is also called a {\em pseudo-circle arrangement}).
We use our definition as in our case any family of homothets of a convex set forms a pseudo-disk arrangement, see, e.g., \cite{Ma2000}.}

\begin{conjecture}\label{conj:pseudo} There is an $m$ such that for any pseudo-disk arrangement any finite set of points admits a $3$-coloring such that any pseudo-disk that contains at least $m$ points contains two points with different colors.
\end{conjecture}

In Conjecture \ref{conj:pseudo} the constant $m$ might in fact be quite a small number.
In an earlier version of this paper we have conjectured that $m=3$ might be sufficient; this, however, has been disproved by G\'eza T\'oth (personal communication); $m=4$ is still possible.

Conjecture \ref{conj:pseudo} also has a natural dual counterpart.

\begin{conjecture}\label{conj:dual} There is an $m$ such that the members of any pseudo-disk arrangement admit a $3$-coloring such that any point that is contained in at least $m$ pseudo-disks is contained in two pseudo-disks with different colors.
\end{conjecture}

We believe that these are fundamental problems about geometric hypergraphs, and find it quite surprising that they have not been studied much.

The rest of this paper is organized as follows.
In the rest of this section, we give an overview of related results.
In Section \ref{sec:delaunay} we give the definition and basic properties of generalized Delaunay triangulations.
In Section \ref{sec:main3} we prove Theorem \ref{thm:main3}, using the proof method of \cite{AKV}.
In Section \ref{sec:highdim} we study the higher dimensional variants of the problem and present some constructions.
In Section \ref{sec:further} we briefly discuss further related topics.

\subsection*{Previous results.}

Most earlier papers on colorings and geometric ranges focused not on proper colorings, but on {\em polychromatic colorings} and its dual, {\em cover-decomposition}.
In the polychromatic $k$-coloring problem our goal is to color the points of some finite $S$ with $k$ colors such that every member of some family \F that contains many points of $S$ contains {\em all} $k$ colors.
Gibson and Varadarajan \cite{GV10} have shown that for every convex polygon $D$ there is a $c_D$ such that every finite set of points can be $k$-colored such that any translate of $D$ that contains at least $m_k=c_D k$ points contains all $k$ colors. 
Whether such a polychromatic $k$-coloring exists for homothetic copies of convex polygons for any $m_k$ is an open problem, which would be a significant strengthening of our Theorem ~\ref{thm:main3}.
This conjecture has only been proved in a series of papers for triangles \cite{CKMUtri,CKMUoct,KP11,selfcover,KP12,octantnine} and very recently \cite{AKV} for squares using the Four Color Theorem.
The derived upper bound on $m_k$ is polynomial in $k$ in both cases.
It is, however, conjectured in a much more general setting \cite{PPhD} that whenever $m_k$ exists, it is linear in $k$, just like for the translates of convex polygons.
This is also known for {\em axis-parallel bottomless rectangles}:\footnote{A bottomless rectangle is a planar set of the form $\{(x,y)\mid a\le x\le b, y\le c\}$.} any finite set of points can be $k$-colored such that any axis-parallel bottomless rectangle that contains at least $m_k=3k-2$ points contains all $k$ colors \cite{A+13}.
(The value of $m_k$ is known to be optimal only for $k=2$ \cite{Khalf}.)
Their proof reduces the problem to coloring a one-dimensional {\em dynamic point set} with respect to intervals, which turns the problem into a variant of online colorings.
We will not introduce these notions here; for some related results, see \cite{CKMUoct,nathan,octantnine}.

The dual notion of polychromatic colorings is {\em cover-decomposition}.
In the cover-decom\-po\-si\-tion problem we are given some finite family \F that covers some region $m_k$-fold (i.e., each point of the region is contained in at least $m_k$ members of \F) and our goal is to partition \F into $k$ families that each cover the region.
By considering the respective underlying incidence hypergraphs in the polychromatic coloring and in the cover-decomposition problems, one can see that they are about colorings of dual hypergraphs.\footnote{The dual of a hypergraph $\HH=(V,\E)$ is the hypergraph $\overline{\HH}$ with vertex set $\E$, edge set $V$, and with the incidences reversed, i.e., in $\overline{\HH}$ a vertex corresponding to $e\in \E$ is contained in the edge corresponding to $v\in V$ if and only if $v$ is contained in $e$ in \HH.}
In fact, the two problems are equivalent for translates of a given set, as the following observation shows.

\begin{observation}[Pach \cite{P80}] For any set $D$, if \HH is the inclusion hypergraph of some points and some translates of $D$, then the dual hypergraph of \HH is also such an inclusion hypergraph.
\end{observation}

Combining this with the result of Gibson and Varadarajan \cite{GV10}, we get that for every convex polygon $D$ there is a $c_D$ such that if \F is an ($m_k=c_D k$)-fold covering of a region by the translates of $D$, then \F can be decomposed into $k$ coverings of the same region.
It follows from the proofs about polychromatic $k$-colorings for triangles that the same holds for coverings by the homothets of a triangle, with a polynomial bound on $m_k$ (this function is slightly weaker than what is known for the polychromatic $k$-coloring problem).
By homothets of other convex polygons, however, surprisingly for any $m$ it is possible to construct an indecomposable $m$-fold covering \cite{Kovindec}.
The homothets of the square are the only family which is known to behave differently for polychromatic coloring and cover-decomposition.

Proper colorings of (primal and dual) geometric hypergraphs have been first studied systematically in \cite{Khalf}, for halfplanes and axis-parallel bottomless rectangles, proving several lower and upper bounds.
Other papers mainly studied the dual variant of our question.
Smorodinsky \cite{Smo07} has shown that any pseudo-disk family can be colored with a bounded number of colors such that every point covered at least twice is covered by at least two differently colored disks.
He also proved that $4$ colors are sufficient for disks, and later this was generalized by Cardinal and Korman \cite{CK} to the homothetic copies of any convex body.
Smorodinsky has also shown that any family of $n$ axis-parallel rectangles can be colored with $O(\log n)$ colors such that every region covered at least twice is covered by at least two differently colored rectangles.
This was shown to be optimal by Pach and Tardos \cite{PaT10}; they proved that there is a $C$ such that for every $m$ there is a family of $n$ axis-parallel rectangles such that for any ($C\frac{\log n}{m\log m}$)-coloring of the family there is a point covered by exactly $m$ rectangles, all of the same color.
It was shown by Chen et al. \cite{CPST09}, answering a question of Brass, Moser and Pach \cite{BMP}, that for every $c$ and $m$ there is a finite point set $S$ such that for every $c$-coloring of $S$ there is an axis-parallel rectangle containing $m$ points that are all of the same color.
This latter construction is the closest to the problem that we study.
It also shows why the pseudo-disk property is crucial in Conjecture \ref{conj:pseudo}.

Rotation invariant families have also been studied.
It was shown in \cite{PTT} using the Hales-Jewett theorem \cite{HJ} that for every $c$ and $m$ there is a finite planar point set $S$ such that for every $c$-coloring of $S$ there is a line containing $m$ points that are all of the same color.
Using duality, they have also shown that this implies that for every $c$ and $m$ there is a finite collection of lines such that for every $c$-coloring of the lines there is a point covered by exactly $m$ lines, all of the same color.
Halfplanes, on the other hand, behave much more like one-dimensional sets and admit polychromatic colorings.
It was shown in \cite{SY12}, improving on earlier results \cite{A+08b,Khalf,PT}, that any finite point set can be $k$-colored such that any halfplane that contains $m_k=2k-1$ points contains all $k$ colors, and this is best possible.
In the dual, they have shown that any finite set of halfplanes can be $k$-colored such that any point that is covered by at least $m_k=3k-2$ halfplanes is covered by all $k$ colors.
This bound is not known to be best possible, except for $k=2$ \cite{F10}.
Except for this last sharpness bound, all other results were extended to pseudo-halfplanes in \cite{KPpseudo}.

%%%%%%%%%%%%%%%%%%%%%%%%%%%%%%%%%%%%%%%%%%%%%%%%%%%%%%%%%%%%%%%%%%%%
\section{Generalized Delaunay triangulations}\label{sec:delaunay}
%%%%%%%%%%%%%%%%%%%%%%%%%%%%%%%%%%%%%%%%%%%%%%%%%%%%%%%%%%%%%%%%%%%%

With a slight perturbation of the points, it is enough to prove Theorem \ref{thm:main3} (or any similar statement) for the case when the points are in a {\em general position with respect to} the convex polygon $D$ in the sense that no two points are on a line parallel to a side of $D$ and no four points are on the boundary of a homothet of $D$.
In the following, we always suppose that our point set $S$ is in general position with respect to $D$.
We will also suppose that $D$ is open - this does not alter the validity of the statements and makes some of the arguments simpler to present.

We say that a halfplane $H$ is {\em supporting $D$ at a side $ab$} of $D$ if $H$ contains $D$ and $ab$ is on the boundary of $H$.
A point $s$ from some set is {\em extremal} (for a side $ab$) if a translate of a halfplane supporting $D$ at a side $ab$ contains $s$ but no other point of the set.

We define a plane graph whose vertices are the points of $S$, called the {\em generalized Delaunay triangulation} of $S$ with respect to $D$, and we denote it by $\DT_D(S)$, or when clear from the context, simply by \DT.
As it leads to no confusion, we will not differentiate between the points and their associated vertices.
Two points of $S$ are connected by a straight-line edge in \DT if there is a homothet of $D$ that contains only them from $S$.
It follows \cite{BCCS10,Kvoronoi} that \DT is a well-defined connected plane graph whose inner faces are triangles.
We recall a few simple statements about \DT, most of which also appeared in \cite{AKV}.

\begin{proposition}\label{prop:connected}
If $D'$ is a homothet of $D$, the points $D'\cap S$ induce a connected subgraph of $\DT_D(S)$.
\end{proposition}

\begin{corollary}[\cite{AKV}]\label{cor:split}
If $D'$ is a homothet of $D$ and $e$ is an edge of \DT that crosses $D'$ and splits it into two parts, then one of these parts does not contain any point from $S$.
\end{corollary}

%We continue with a proposition that is quite similar to a statement of \cite{AKV}.

%\begin{proposition}\label{prop:corner}
%Suppose that $D'$ and $D_y$ are two homothets of $D$, $x,y,y',z\in D_y\cap S$ and $y,y'\in D'$ but $x,z\notin D'$, $x$ and $z$ are neighbors of $y$ in \DT, and for one of the two cones whose sides are the halflines starting from $y$ as $yx$ and $yz$, denoted by $C$, we have $C\cap D'\subset D_y$ and $y'\in C\cap D'$.
%(See Figure \ref{fig:corner} for an illustration.)
%Then $y'$ has a neighbor in \DT that is contained in $D'\cap D_y$.
%\end{proposition}
%\begin{proof} Using Proposition \ref{prop:connected}, $y'$ has a neighbor in $D'$.
%Since \DT is planar, this neighbor must be in $C\cap D'\subset D'\cap D_y$. 
%\end{proof}

%fig:corner used to be here

%%%%%%%%%%%%%%%%%%%%%%%%%%%%%%%%%%%%%%%%%%%%%%%%%%%%%%%%%%%%%%%%%%%%
\section{Framework}\label{sec:framework}
%%%%%%%%%%%%%%%%%%%%%%%%%%%%%%%%%%%%%%%%%%%%%%%%%%%%%%%%%%%%%%%%%%%%

In this section we outline the main idea behind the proof of Theorem \ref{thm:main3}.
As discussed in Section \ref{sec:delaunay}, we can suppose that $S$ is in general position with respect to $D$, and we can consider the generalized Delaunay triangulation $\DT=\DT_D(S)$.
We will take an initial coloring of $S$ that has some nice properties.
More specifically, we need a $3$-coloring for which the assumptions of the following lemma hold for $c=3$ and for some constant $t$ that only depends on $D$.

\begin{lemma}\label{lem:main}
For every convex polygon $D$ for every $c$ and $t$ there is an $m$ such that\\
if for a $c$-coloring of a point set $S$ and a set of points $R\subset S$ and for every homothet $D'$
\begin{itemize} 
\item[(i)]  if $D'\cap S$ is monochromatic with at least $t$ vertices, $D'$ contains a point of $R$, 
\item[(ii)] if $D'$ contains $t$ points from $R$ colored with the same color, $D'$ also contains a point from $S\setminus R$ that has the same color,
\end{itemize}
then there is a $c$-coloring of $S$ such that no homothet that contains at least $m$ points of $S$ is monochromatic.
\end{lemma}

To prove Lemma \ref{lem:main}, we use the following theorem about the so-called {\em self-coverability} of convex polygons.

\begin{theorem}[\cite{selfcover}]\label{thm:selfcover}
Given a closed convex polygon $D$ and a collection of $k$ points in its interior, we can take $c_Dk$ homothets of $D$ whose union is $D$ such that none of the homothets contains any of the given points in its interior, where $c_D$ is a constant that depends only on $D$.
\end{theorem}

\begin{figure}[h]
	\centering
	%\includegraphics[width=5cm]{prop_twohomot}
	%\hspace{1cm}
	\includegraphics[width=6cm]{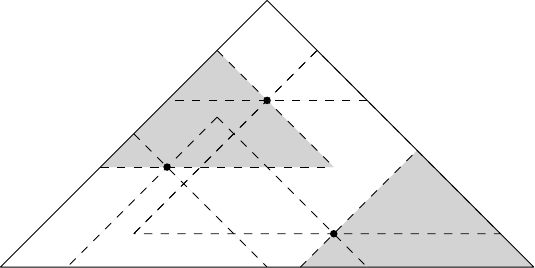}
	\caption{Illustration %for Proposition \ref{prop:corner} on the left and
		for Theorem \ref{thm:selfcover} for triangles %to the right
		(shading is only to improve visibility.)}
	%\label{fig:corner}
	\label{fig:selfcover}
\end{figure}

\begin{proof}[Proof of Lemma \ref{lem:main}]
The proper $c$-coloring will be simply taking the $c$-coloring given in the hypothesis, and recoloring each vertex in $R$ arbitrarily to a different color.
Now we prove the correctness of this new coloring.
Let $D'$ be a homothet of $D$ containing at least $m$ points (where $m$ is to be determined later).

Suppose first that $D'$ contains $m\ge ct$ points from $R$.
Using the pigeonhole principle, $D'$ contains at least $t$ points from $R$ that originally had the same color.
Using (ii), $D'$ will have a point both in $R$ and in $S\setminus R$ that had the same color.
These points will have different colors after the recoloring, thus $D'$ will not be monochromatic.

Otherwise, suppose that $D'$ contains $m$ points of which less than $ct$ are from $R$.
Apply Theorem \ref{thm:selfcover} with $D'$ and $R'=D'\cap R$.
This gives $c_Dct$ homothets (where $c_D$ comes from Theorem \ref{thm:selfcover}), each of which might contain at most three points on their boundaries (which include the points from $R'$), thus by the pigeonhole principle at least one homothet, $D''$, contains no points from $R$ and at least $\frac{m-3c_Dct}{c_Dct}$ points from $S\setminus R$.
If we set $m=c_Dct(t+3)$, this is at least $t$.
Thus, by $(i)$, $D''$ was not monochromatic before the recoloring.
As the recoloring does not affect points in $S\setminus R$, after the recoloring $D''$ (and so also $D'$) still contains two points that have different colors.
Thus $m=c_Dct(t+3)$ is a good choice for $m$ in both cases.
\end{proof}

Therefore, to prove Theorem \ref{thm:main3}, we only need to show that we can find a coloring with three colors that satisfy the conditions of Lemma \ref{lem:main} for some $t$.

%%%%%%%%%%%%%%%%%%%%%%%%%%%%%%%%%%%%%%%%%%%%%%%%%%%%%%%%%%%%%%%%%%%%
\section{Proof of Theorem \ref{thm:main3}}\label{sec:main3}
%%%%%%%%%%%%%%%%%%%%%%%%%%%%%%%%%%%%%%%%%%%%%%%%%%%%%%%%%%%%%%%%%%%%

In this section we prove Theorem \ref{thm:main3}, that is, we show that for every convex polygon $D$ there is an $m$ such that any finite set of points $S$ admits a $3$-coloring such that there is no monochromatic homothet of $D$ that contains at least $m$ points.
If one could find a $3$-coloring where every monochromatic component of \DT is bounded, then that would immediately prove Theorem \ref{thm:main3}.
This, however, is not true in general \cite{KMRV97}, only for bounded degree graphs \cite{EJ14}, but the \DT can have arbitrarily high degree vertices for any convex polygon, thus we cannot apply this result.
Instead, we use the following result (whose proof is just a couple of pages).

\begin{theorem}[Poh \cite{Poh}; Goddard \cite{Goddard}]\label{thm:pathcol}
The vertices of any planar graph can be $3$-colored such that every monochromatic component is a path.
\end{theorem}

To prove Theorem \ref{thm:main3}, apply Theorem \ref{thm:pathcol} to \DT to obtain a $3$-coloring where every monochromatic component is a path.
It follows from Lemma \ref{lem:main} that it is sufficient to show that for $t=4n+12$ (where $n$ denotes the number of sides of $D$) there is a set of points $R\subset S$ for which

\begin{itemize}
\item[(i)] for every homothet $D'$ if $D'\cap S$ is monochromatic with at least $t$ vertices, $D'$ contains a point of $R$, 
\item[(ii)] for every homothet $D'$ if $D'$ contains $t$ points from $R$ colored with the same color, $D'$ also contains a point from $S\setminus R$ that has the same color.
\end{itemize}

Now we describe how to select $R$.
First, partition every monochromatic path that has at least $t$ vertices into subpaths, called {\em sections}, such that the number of vertices of each section is at least $\frac t4$ but at most $\frac t2$.
We call such a section {\em cuttable} if there is a monochromatic homothet of $D$ that contains all of its points.
$R$ will consist of exactly one point from each cuttable section.
These points are selected arbitrarily from the non-extremal points of each section, except that they are required to be non-adjacent on their monochromatic path.
Since each section has at most two end points and $n$ extremal points, we can select such a point from each section if $\frac t4\ge n+3$.
For an $r\in R$ we denote its section by $\sigma_r$ and a (fixed) monochromatic homothet containing $\sigma_r$ by $D_r$.

Now we prove that $R$ satisfies the requirements (i) and (ii).

To prove (i), suppose that a homothet $D'$ is monochromatic with at least $t$ vertices.
Using Proposition \ref{prop:connected}, the subgraph induced on these vertices is connected.
As any monochromatic connected component is a path, $D'$ contains at least $t$ consecutive vertices of a monochromatic path, and thus also a section.
Because of $D'$ this section is cuttable, and thus contains a point of $R$.

To prove (ii), suppose that a homothet $D'$ contains $t$ points from $R$ colored with the same color, red.
Denote these points by $R'$.
For each $r\in R'$, the neighbors of $r$ in $\sigma_r$ are red but not in $R$, thus they must be outside $D'$, or otherwise (ii) holds and we are done.
Denote the geometric embedding of the two edges adjacent to $r$ in $\sigma_r$ by $\Lambda_r$.
Therefore, $\Lambda_r$ will intersect the boundary of $D'$ in two points for each $r\in R'$.
We claim that these two intersection points usually fall on the same side of $D'$, i.e., they are not separated along the boundary by a vertex.

\begin{proposition}\label{prop:most2side} Both intersection points of $\Lambda_r$ and the boundary of $D'$ are on the same side of $D'$ for all but at most $n$ points of $r\in R'$.
\end{proposition}
\begin{figure}[h]
    \centering
		\includegraphics[width=8cm]{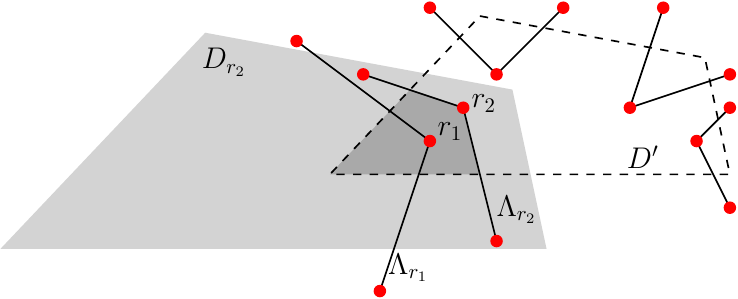}
	\caption{Proof of Proposition \ref{prop:most2side}.}
	\label{fig:twoside}
\end{figure}
\begin{proof}
Suppose that there are more than $n$ points $r\in R'$ for which $\Lambda_r$ intersects $D'$ in two sides.
For each such point $r\in R'$, for (at least) one of the two (one convex and one non-convex) cones whose sides are the halflines starting in $\Lambda_r$, denoted by $C_r$, we have $C_r\cap D'\subset D_r$.
Since the intersection $C_r\cap D'$ is a connected curve, it contains a vertex of $D'$.
Using the pigeonhole principle, there are two points, $r_1,r_2\in R'$, such that $C_{r_1}$ and $C_{r_2}$ contain the same vertex of $D'$.
(See Figure \ref{fig:twoside}.)
As $\Lambda_{r_1}\cap \Lambda_{r_2}=\emptyset$, we have (without loss of generality) $r_1\in C_{r_2}$, which also implies $r_1\in D_{r_2}$.
But using Proposition \ref{prop:connected}, $r_1$ must have a neighbor in $D'$.
Since this neighbor is also in $D_{r_2}$, it has to be red.
%In this case $r_1$ and $r_2$ would be in the configuration described in Proposition \ref{prop:corner}, one playing the role of $y$, the other of $y'$. 
%The neighbor of $y'$ in $D'\cap D_y$ given by Proposition \ref{prop:corner} must also be red, as it is contained in $D_y$.
As the red neighbors of any red point of $R$ are not in $R$, we have found a red point from $S\setminus R$ in $D'$, proving (ii).
\end{proof}

Divide the points $r\in R'$ for which $\Lambda_r$ intersects only one side of $D'$ into $n$ groups, $R_1',\ldots,R_n'$, depending on which side is intersected.
By the pigeonhole principle there is a group, $R_i'$, that contains at least $\frac{t-n}n\ge 3$ points.
Suppose without loss of generality that the side $ab$ intersected by $\Lambda_r$ for $r\in R_i'$ is horizontal, bounding $D'$ from below.
For each $r\in R_i'$, fix and denote by $x_r$ a point from $\sigma_r$ whose $y$-coordinate is larger than the $y$-coordinate of $r$.
(Such a point exists because no $r\in R$ is extremal in $\sigma_r$.)
Denote the path from $r$ to $x_r$ in $\sigma_r$ by $P_r$, and the neighbor of $r$ in $P_r$ by $q_r$.

The geometric embedding of $P_r$ starts above $ab$ with $r$, then goes below $ab$ as $q_r\notin D'$, and finally $x_r$ is again above the line $ab$.
Denote the first intersection (starting from $r$) of the embedding of the path $P_r$ with the line $ab$ by $\alpha_r=\bar{rq_r}\cap\bar{ab}$, and the next intersection by $\beta_r$.
Since $|R_i'|\ge 3$, without loss of generality, there are $r_1,r_2\in R_i'$ such that $\beta_{r_1}$ is to the left of $\alpha_{r_1}$ and $\beta_{r_2}$ is to the left of $\alpha_{r_2}$.
For readability and simplicity, let $x_i=x_{r_i}$, $P_i=P_{r_i}$, $q_i=q_{r_i}$, $\alpha_i=\alpha_{r_i}$, $\beta_i=\beta_{r_i}$.

\begin{figure}
    \centering
		\includegraphics[width=5cm]{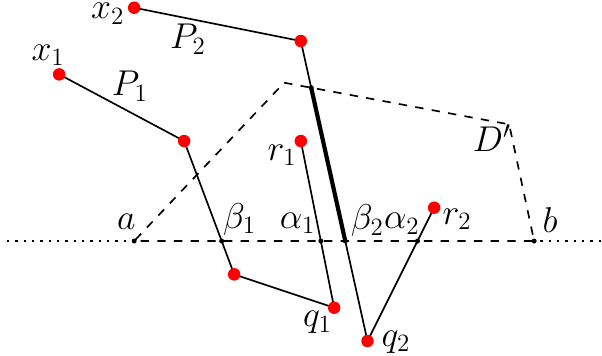}
		\hspace{1cm}
    \includegraphics[width=7cm]{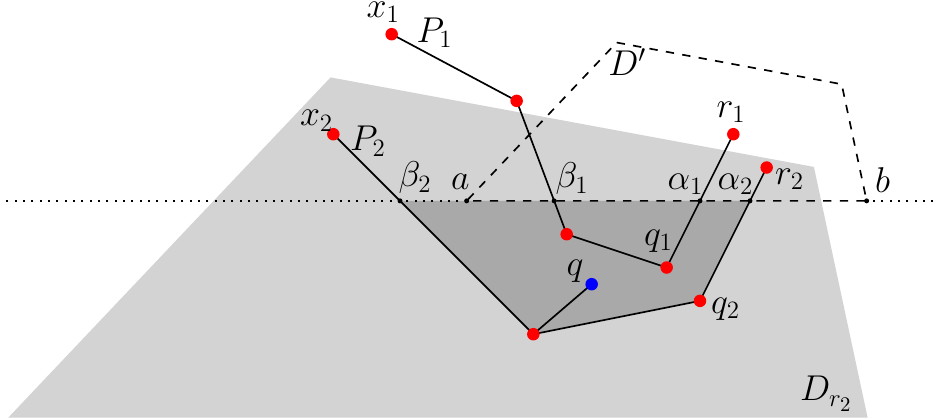}
	\caption{The two cases at the end of the proof of Theorem \ref{thm:main3}. To the left, $\beta_2$ is to the right of $\alpha_1$, the part of the edge splitting $D'$ is bold. To the right, $\beta_2$ is to the left of $\alpha_1$, the shaded regions must contain $q_1$.}
	\label{fig:r1r2}
\end{figure}

Without loss of generality suppose that $\alpha_1$ is to the left of $\alpha_2$.
Recall that $P_2$ contains only red points, of which only $r_2$ is in $R$.
Therefore, no other vertex of $P_2$ can be in $D'$.
If $\beta_2$ is to the right of $\alpha_1$, then one of the edges of $P_2$ would separate $r_1$ and $r_2$ in the sense described in
Corollary \ref{cor:split}.
(See Figure \ref{fig:r1r2}.)
As this cannot happen, $\beta_2$ is to the left of $\alpha_1$.

This implies that $q_1\notin P_2$ is in the convex hull of $P_2$ below the $ab$ line.
Take the point $q\in S\setminus P_2$ with the smallest $y$-coordinate such that $q$ is in the convex hull of $P_2$ below the $ab$ line.
As $q$ is not an extremal point of $S$, it is connected in \DT to some point in $S$ whose $y$-coordinate is smaller (because the faces of \DT are triangles).
By the definition of $q$, this neighbor must be in $P_2$.
As the end vertices of $P_2$, $r_2$ and $x_2$, are above the $ab$ line, $q$ is connected to an inner vertex of a monochromatic red path.
Since every monochromatic component is a path, $q$ cannot be red.
The homothet $D_{r_2}$ contains the red vertices of $P_2$ and thus all the points in the convex hull of $P_2$.
But $D_{r_2}$ is monochromatic, so it cannot contain the non-red point $q$, a contradiction.

This finishes the proof of Theorem \ref{thm:main3}.

%%%%%%%%%%%%%%%%%%%%%%%%%%%%%%%%%%%%%%%%%%%%%%%%%%%%%%%%%%%%%%%%%%%%
\section{Higher dimensions}\label{sec:highdim}
%%%%%%%%%%%%%%%%%%%%%%%%%%%%%%%%%%%%%%%%%%%%%%%%%%%%%%%%%%%%%%%%%%%%

In this section we study the following natural extension of the problem to higher dimensions.
Given a finite set of points $S\in \R^d$ and a family \F, can we $c$-color $S$ such that every $F\in \F$ contains at least two colors?
First we show that for $d\ge 4$ this is not even possible for hextants.
Define a (positive) {\em hextant} in $\R^4$ as the set of points $\{(x,y,z,w)\mid x\ge x_0, y\ge y_0, z\ge z_0, w\ge w_0\}$ for some real numbers $x_0,y_0,z_0,w_0$.
Cardinal noticed that hextants can simulate the axis-parallel rectangles of an appropriate subplane of $\R^4$ and thus the following holds.

\begin{theorem}[Cardinal\footnote{Cardinal (personal communication) stated this for $c=2$ using the same reduction based on \cite{PTT} about axis-parallel rectangles; Theorem \ref{thm:hex2rec} is only more general because we use a stronger result \cite{CPST09} about axis-parallel rectangles.}]\label{thm:hex2rec} For any $c$ and $m$ there is a finite point set $S$ such that for every $c$-coloring of $S$ there is a hextant that contains exactly $m$ points of $S$, all of the same color.
\end{theorem}
\begin{proof} As mentioned in the introduction, Chen et al. \cite{CPST09} have shown that for any $c$ and $m$ there is a finite planar point set $S$ such that for every $c$-coloring of $S$ there is an axis-parallel rectangle that contains exactly $m$ points of $S$, all of the same color.
Place this construction on the $\Pi=\{(x,y,z,w)\mid x+y=0$, $z+w=0\}$ subplane of $\R^4$.
A hextant $\{(x,y,z,w)\mid x\ge x_0, y\ge y_0, z\ge z_0, w\ge w_0\}$ intersects $\Pi$ in $\{(x,y,z,w)\mid x_0\le x=-y\le -y_0, z_0\le z=-w\le -w_0\}$, which is a rectangle whose sides are parallel to the lines $\{x+y=0, z=w=0\}$ and $\{x=y=0, z+w=0\}$, respectively.
Taking these perpendicular lines as axes, thus any ``axis-parallel'' rectangle of $\Pi$ is realizable by an appropriate hextant, and the theorem follows.
\end{proof}

Kolja Knauer observed\footnote{As our referee, disproving a bold conjecture we made in an earlier version of this paper.} that all axis-parallel rectangles of a subplane of $\R^3$ can be cut out in a similar way by the homothets of a (regular) tetrahedron.
Indeed, let $\Delta$ be the tetrahedron whose vertices are $(1,0,1),(-1,0,1),(0,-1,-1),(0,1,-1)$.
Let $\Delta_h$ be a translate of $\Delta$ by $-1<h<1$ parallel to the $z$-axis.
The intersection of $\Delta_h$ with the plane $\Pi=\{(x,y,z)\mid z=0\}$ yields an axis-parallel rectangle $R_h=\Delta_h\cap \Pi$.
The ratio of the sides of $R_h$ depends on $h$, and can take any value, as it tends to $\pm \infty$ as $h\to \pm 1$.
It follows that by taking a homothetic scaling of $D_h$, we can obtain any axis-parallel rectangle.
Just like in the proof of Theorem \ref{thm:hex2rec},
we obtain by \cite{CPST09} that for any $c$ and $m$ there is a finite point set $S\subset \R^3$ such that for every $c$-coloring of $S$ there is a homothet of $\Delta$ that contains exactly $m$ points of $S$, all of the same color.

%\begin{conjecture}\label{conj:3d} For every convex set $D\subset \R^3$ there is an $m$ such that any finite set of points admits a $4$-coloring such that any homothetic copy of $D$ that contains at least $m$ points contains at least two colors.
%\end{conjecture}
For balls, however, we do not know of any counterexamples, even though a $3$-dimensional {\em Delaunay triangulation} of any number of points might induce a complete graph (for a recent proof, see \cite{GP16}).
We find it quite surprising that while in the plane convex polygons admit polychromatic colorings and disks do not, in the space it might be vice versa.
We could only prove the following weaker statement.
%As a (probably different) anonymous referee called our attention to it,\footnote{We have claimed the opposite in the first version of our manuscript.} a $3$-dimensional {\em Delaunay triangulation} of any number of points might induce a complete graph (for a recent proof, see \cite{GP16}), so it is not even clear that Conjecture \ref{conj:3d} is true with any number of colors instead of $4$, unlike it was in $2$ dimensions.

\begin{theorem}\label{thm:3d} For every $m$ there is a finite set of points $S\in \R^3$ such that for any $3$-coloring of $S$ there is a unit ball that contains exactly $m$ points of $S$, all of the same color.
\end{theorem}

Earlier such a construction with unit balls was only known for $2$-colorings \cite{PTT}.
For $2$-colorings the analogue of Theorem \ref{thm:3d} was also shown to hold when the family is the translates of any polyhedron instead of unit balls \cite{concave}.
The only known positive result is that for {\em octants} any finite set of points can be $2$-colored such that any octant that contains at least $9$ points contains both colors \cite{KP11,octantnine}.
We do not, however, know the answer for $3$-colorings and the translates or homothets of polyhedra.

The rest of this section contains a sketch of the proof of Theorem \ref{thm:3d}.
The reason why we only sketch the proof is that it is a simple modification of the planar construction with similar properties for unit disks from \cite{PP}.\\

{\bf Abstract hypergraph.}

First we define the abstract hypergraph that will be realized with unit balls.
It is a straight-forward generalization of the hypergraph defined first in \cite{concave}.
Instead of a single parameter, $m$, the induction will be on three parameters, $k,l$ and $m$.
For any $k,l,m$ we define the (multi)hypergraph $\HH(k,l,m)=(V(k,l,m),\E(k,l,m))$ recursively.
The edge set $\E(k,l,m)$ will be the disjoint union of three sets, $\E(k,l,m)=\E_1(k,l,m)\cupdot \E_2(k,l,m) \cupdot \E_3(k,l,m)$.
All edges belonging to $\E_1(k,l,m)$ will be of size $k$, all edges belonging to $\E_2(k,l,m)$ will be of size $l$, and all edges belonging to $\E_3(k,l,m)$ will be of size $m$.
We will prove that in every $3$-coloring of $\HH(k,l,m)$ with colors $c_1,c_2$ and $c_3$ there will be an edge in $\E_i(k,l,m)$ such that all of its vertices are colored $c_i$ for some $i\in \{1,2,3\}$.
If $k=l=m$, we get an $m$-uniform hypergraph that cannot be properly $3$-colored.

\begin{figure}[h]
    \centering
		\includegraphics[width=4cm]{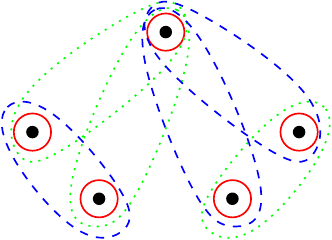}
		\hspace{1cm}
    \includegraphics[width=8cm]{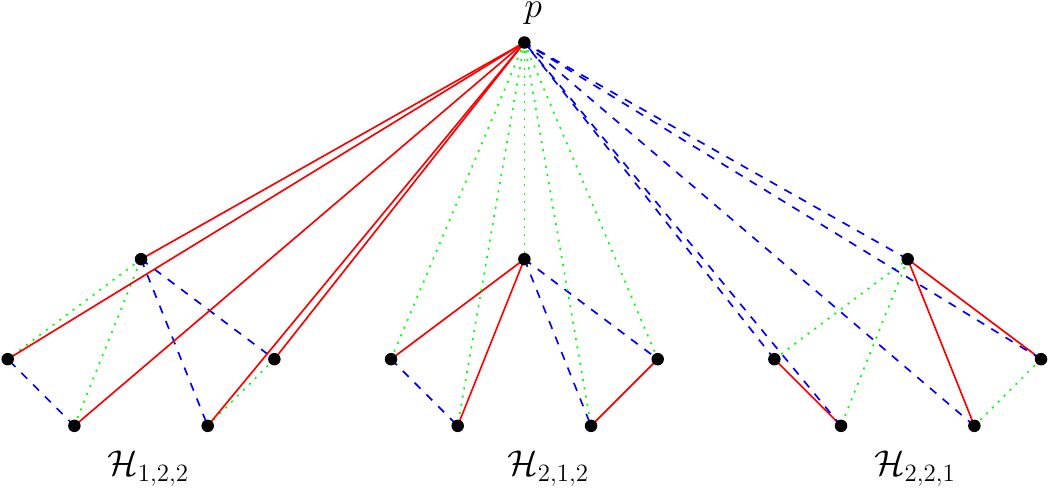}
	\caption{$\HH(1,2,2)$ drawn with sets (left) and $\HH(2,2,2)$ drawn as graph (right). Different colors represent the edges from the different families $\E_i$.}
	\label{fig:hh}
\end{figure}

Now we give the recursive definition.
Define $\HH(1,1,1)$ as a hypergraph on one vertex with three edges containing it, with one edge in each of $\E_1(1,1,1)$, $\E_2(1,1,1)$ and $\E_3(1,1,1)$.
If at least one of $k,l,m$ is bigger than $1$, define $\HH(k,l,m)$ recursively from $\HH(k-1,l,m)$, $\HH(k,l-1,m)$, $\HH(k,l,m-1)$ by adding a ``new'' vertex $p$ as follows.
$$V(k,l,m)=V(k-1,l,m)\cupdot V(k,l-1,m)\cupdot V(k,l,m-1)\cupdot \{p\}.$$

If $k=1$, then $\E_1(1,l,m)=\{\{v\} \,:\, v\in V(1,l,m)\}$, otherwise
$$\E_1(k,l,m)=\{e\cup \{p\} \,:\, e\in \E_1(k-1,l,m)\} \cupdot \E_1(k,l-1,m) \cupdot \E_1(k,l,m-1).$$

Similary, if $l=1$, then $\E_2(k,1,m)=\{\{v\} \,:\, v\in V(k,1,m)\}$, otherwise
$$\E_2(k,l,m)=\{e\cup \{p\} \,:\, e\in \E_2(k,l-1,m)\} \cupdot \E_2(k-1,l,m) \cupdot \E_2(k,l,m-1),$$

and if $m=1$, then $\E_3(k,l,1)=\{\{v\} \,:\, v\in V(k,l,1)\}$, otherwise
$$\E_3(k,l,m)=\{e\cup \{p\} \,:\, e\in \E_3(k,l,m-1)\} \cupdot \E_3(k-1,l,m) \cupdot \E_3(k,l-1,m).$$

\begin{lemma}\label{lem:hipergraf}
In every $3$-coloring of $\HH(k,l,m)$ with colors $c_1,c_2$ and $c_3$ there is an edge in $\E_i(k,l,m)$ such that all of its vertices are colored $c_i$ for some $i\in \{1,2,3\}$.
Therefore, $\HH(k,l,m)$ has no proper $3$-coloring.
\end{lemma}
The proof is a simple modification of the respective statement from \cite{concave}.
\begin{proof} If $k=l=m=1$, the statement holds.
Otherwise, suppose without loss of generality that the color of $p$ is $c_1$.
If $k=1$, we are done as $\{p\}\in \E_1(1,l,m)$.
Otherwise, consider the copy of $\HH(k-1,l,m)$ contained in $\HH(k,l,m)$.
If it contains an edge in $\E_2(k-1,l,m)$ or $\E_3(k-1,l,m)$ such that its vertices are all colored $c_2$ or all colored $c_3$, respectively, we are done.
Otherwise, it contains an $e\in \E_1(k-1,l,m)$ such that its vertices are all colored $c_1$.
But then all the vertices of $(e\cup \{p\})\in \E_1(k,l,m)$ are also all colored $c_1$, we are done.
\end{proof}

{\bf Geometric realization.}

Now we sketch how to realize $\HH(k,l,m)$ by unit balls in $\R^3$.
The construction will build on the construction of \cite{PP}, where the edges belonging to $\E_1(k,l,1)\cupdot \E_2(k,l,1)$ of $\HH(k,l,1)$ were realized by unit disks.

The vertices $V(k,l,m)$ will be embedded as a point set, $S(k,l,m)$, and the edge set $\E_i(k,l,m)$ as a collection of unit balls, $\B_i(k,l,m)$, where a point is contained in a ball if and only if the respective vertex is in the respective edge.
All the points of $S(k,l,m)$ will be placed in a small neighborhood of the origin.
The centers of the balls from $\B_1(k,l,m)$, $\B_2(k,l,m)$ and $\B_3(k,l,m)$ will be close to $(0,-1,0)$, $(0,1,0)$ and $(0,0,-1)$, respectively.
The realization of $\HH(1,1,1)$ contains only one point, the origin, and one ball in each family, centered appropriately close to the required center.

\begin{figure}
    \centering
    \includegraphics[width=6cm]{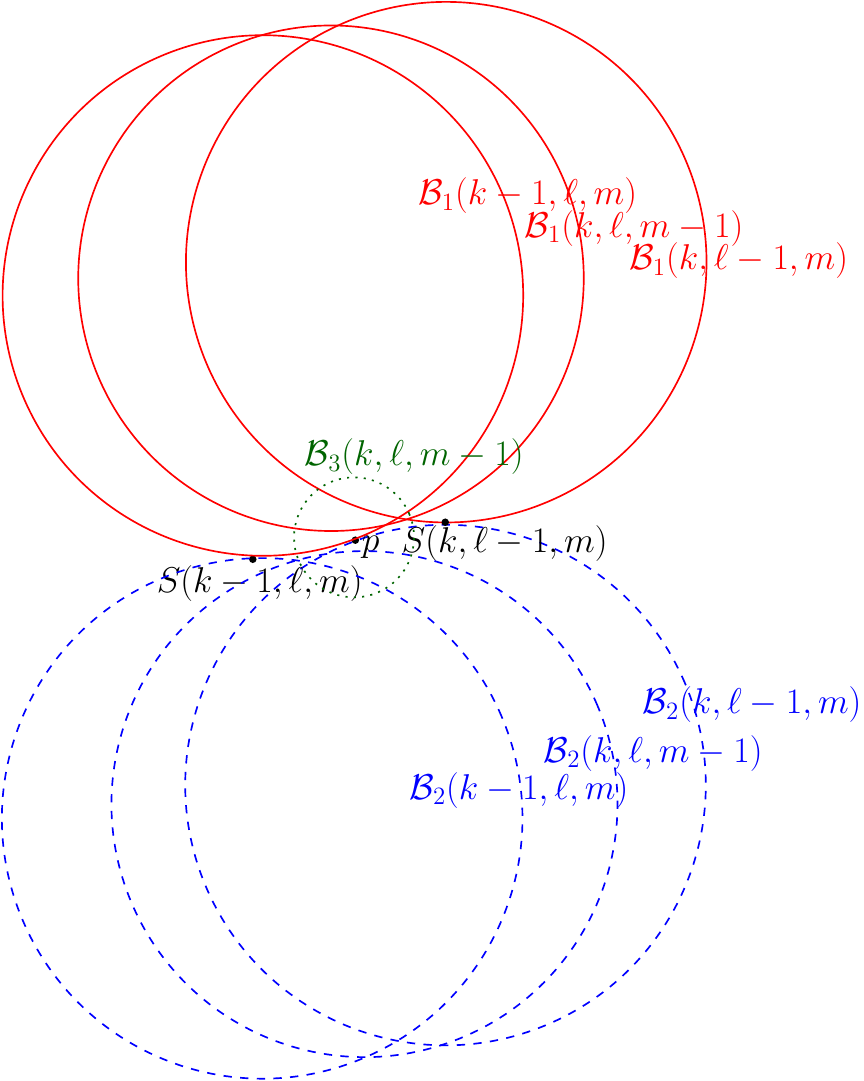}
	\caption{The intersection of $\HH(k,l,m)$ with the $z=0$ plane. Point sets/collections of balls that are at distance $O(\eps^5)$ are represented by a single point/ball. As the balls $\B_3(k-1,l,m)$ intersect in a $O(\eps^5)$ vicinity of $S(k-1,l,m)$ and the balls $\B_3(k,l-1,m)$ intersect in a $O(\eps^5)$ vicinity of $S(k,l-1,m)$, they are not drawn to avoid overcrowding the picture.}
	\label{fig:3dPi}
\end{figure}

Suppose that not all of $k,l,m$ are $1$, and we have already realized the hypergraphs $\HH(k-1,l,m)$, $\HH(k,l-1,m)$ and $\HH(k,l,m-1)$.
Place the new point $p$ in the origin, and shift the corresponding realizations (i.e., the point sets, $S(k-1,l,m)$, $S(k,l-1,m)$ and $S(k,l,m-1)$, and the collection of balls, $\B(k-1,l,m)$, $\B(k,l-1,m)$ and $\B(k,l,m-1)$) by the following vectors, where $\eps=\eps(k,l,m)$ is a small enough number, but such that $\eps(k-1,l,m)$, $\eps(k,l-1,m)$ and $\eps(k,l,m-1)$ are all $O(\eps^5(k,l,m))$.

\begin{itemize}
\item[1.] Shift $\HH(k-1,l,m)$ by $(2\eps-1.5\eps^3,2\eps^2,0)$.
\item[2.] Shift $\HH(k,l-1,m)$ by $(-2\eps+1.5\eps^3,-2\eps^2,0)$.
\item[3.] Shift $\HH(k,l,m-1)$ by $(0,0,2\eps^2)$.
\end{itemize}

For an illustration, see Figure \ref{fig:3dPi}.

\begin{proposition}\label{prop:Hreal} The above construction realizes $\HH(k,l,m)$.
\end{proposition}

The proof of this proposition is a routine calculation, we only show some parts. 

\begin{proof}
	Denote by $o_B$ the center of the ball $B$ and denote by $dist(p,q)$ the Euclidean distance of two points $p,q$.
\begin{enumerate}
\item $p\in B\in \B_1(k-1,l,m)$:
$$dist^2(p,o_B)=(2\eps-1.5\eps^3)^2+(1-2\eps^2)^2+O(\eps^5)=1-2\eps^4+O(\eps^5)<1.$$

\item $p\notin B\in \B_1(k,l-1,m)$: 
$$dist^2(p,o_B)=(2\eps-1.5\eps^3)^2+(1+2\eps^2)^2+O(\eps^5)=1+4\eps^2+O(\eps^3)>1.$$ 

\item $p\notin B\in \B_1(k,l,m-1)$: 
$$dist^2(p,o_B)=1^2+(2\eps^2)^2+O(\eps^5)=1+4\eps^4+O(\eps^5)>1.$$

\item If $s\in S(k,l-1,m)$, then $s\notin B\in \B_1(k-1,l,m)$: 
$$dist^2(s,o_B)=(4\eps-3\eps^3)^2+(1-4\eps^2)^2+O(\eps^5)=1+8\eps^2+O(\eps^3)>1.$$ 

\item If $s\in S(k,l-1,m)$, then $s\notin B\in \B_1(k,l,m-1)$: 
$$dist^2(s,o_B)=(2\eps-1.5\eps^3)^2+(1-2\eps^2)^2+(2\eps^2)^2+O(\eps^5)=1+2\eps^4+O(\eps^5)>1.$$

\item If $s\in S(k,l-1,m)$, then $s\notin B\in \B_3(k,l,m-1)$: 
$$dist^2(s,o_B)=(2\eps-1.5\eps^3)^2+(2\eps^2)^2+(1-2\eps^2)^2+O(\eps^5)=1+2\eps^4+O(\eps^5)>1.$$

\item If $s\in S(k,l,m-1)$, then $s\notin B\in \B_1(k-1,l,m)$: 
$$dist^2(s,o_B)=(2\eps-1.5\eps^3)^2+(1-2\eps^2)^2+(2\eps^2)^2+O(\eps^5)=1+2\eps^4+O(\eps^3)>1.$$

\end{enumerate}

The other incidences can be checked similarly and thus Proposition \ref{prop:Hreal} follows.
\end{proof}

Lemma \ref {lem:hipergraf} and Proposition \ref{prop:Hreal} imply Theorem \ref{thm:3d} by selecting $k=l=m$, therefore this also finishes the proof of Theorem \ref{thm:3d}.

%%%%%%%%%%%%%%%%%%%%%%%%%%%%%%%%%%%%%%%%%%%%%%%%%%%%%%%%%%%%%%%%%%%%
\section{Further remarks}\label{sec:further}
%%%%%%%%%%%%%%%%%%%%%%%%%%%%%%%%%%%%%%%%%%%%%%%%%%%%%%%%%%%%%%%%%%%%

Combining Theorems \ref{thm:main3} and \ref{thm:selfcover}, for any convex polygon, $D$, and for any finite point set, $S$, we can first find a $3$-coloring of $S$ using Theorem \ref{thm:main3} such that every large (in the sense that it contains many points of $S$) homothet of $D$ contains two differently colored points, then using Theorem \ref{thm:selfcover} we can conclude that every very large homothet of $D$ contains many points from at least two color classes, and finally we can recolor every color class separately using Theorem \ref{thm:main3}.
This proves that for every $k$ there is a $3^k$-coloring such that every large homothet of $D$ contains at least $2^k$ colors.
Of course, the colors that we use when recoloring need not be different for each color class, so we can also prove for example that there is a $6$-coloring such that every large homothet of $D$ contains at least $3$ colors.
What are the best bounds of this type that can be obtained?\\

Given a planar graph, $G$, and a pair of paths on three vertices, $uvw$ and $u'vw'$, we say that the paths {\em cross} at $v$ if $u,u',w,w'$ appear in this order around $v$. 
A possible equivalent reformulation of Conjecture \ref{conj:pseudo} is the following.
Is it true that for any planar graph and any pairwise non-crossing collection of its paths on three vertices, \P, there is a $3$-coloring of the vertices such that every path from \P is non-monochromatic?\\

Finally, we would like to draw attention to the study of {\em realizable hypergraphs}.
Unfortunately, {\em planar hypergraphs} are traditionally defined dully as a hypergraph whose (bipartite) incidence graph is planar.
Instead, it would be more natural to define them as the hypergraphs realizable by a pseudo-disk arrangement in the sense that the vertices are embedded as points and the edges as pseudo-disks such that a point is contained in a pseudo-disk if and only if the respective vertex is in the respective edge.
This was done in \cite{BPR13}, where it was proved that such a hypergraph on $n$ vertices can have at most $O(k^2n)$ edges that each contain at most $k$ points, while there can be at most $3n-6$ edges containing exactly two points, matching Euler's bound for planar graphs.
Despite \cite{BPR13}, these hypergraphs received little attention and even simple statements are highly non-trivial; see the recent proof by Kisfaludi-Bak\footnote{\url{http://mathoverflow.net/a/257212/955}.} that the complete $3$-uniform hypergraph on $5$ vertices is not realizable by pseudo-disks.
We believe that these hypergraphs deserve more attention.

\subsubsection*{Acknowledgment}

We would like to thank our anonymous referees for several suggestions that improved the presentation of our results, and to Arnau Padrol for explaining to us the example in \cite{GP16}.

\bibliography{3colwrthomot}
		
\end{document}